\documentclass[a4paper,12pt]{article}

\usepackage[latin1]{inputenc}
\usepackage[T1]{fontenc}
\usepackage[english]{babel}

\usepackage{mathrsfs}
\usepackage{amscd}
\usepackage{amsfonts}
\usepackage{amsmath}
\usepackage{amssymb}
\usepackage{amstext}
\usepackage{amsthm}
\usepackage{amsbsy}

\usepackage{xspace}
\usepackage[all]{xy}
\usepackage{graphicx}
\usepackage{url}
\usepackage{latexsym}

\usepackage{graphicx} 

\usepackage{booktabs} 
\usepackage{array} 
\usepackage{paralist} 
\usepackage{verbatim} 
\usepackage{subfig} 

\usepackage{fancyhdr} 
\usepackage{sectsty}
\allsectionsfont{\sffamily\mdseries\upshape} 

\usepackage[nottoc,notlof,notlot]{tocbibind} 
\usepackage[titles,subfigure]{tocloft} 

\usepackage[textwidth=100pt,textsize=footnotesize,bordercolor=white,color=blue!30]{todonotes}
\usepackage{hyperref} 

\makeatletter
\newcommand*{\rom}[1]{\expandafter\@slowromancap\romannumeral #1@}
\makeatother

\theoremstyle{definition}

\newtheorem{fact}{fact}

\newtheorem{thm}[fact]{Theorem}
\newtheorem{lemma}[fact]{Lemma}

\newtheorem{defini}[fact]{Definition}

\title{A Note on Always Decidable Propositional Forms}
\author{Merlin Carl}

\begin{document}

\maketitle

\begin{abstract}
We ask the following question: If all instantiations of a propositional formula $A(x_1,...,x_n)$ in $n$ propositional variables are decidable in some sufficiently strong recursive theory, does it follow
that $A$ is tautological or contradictory? and answer it in the affirmative. We also consider the following related question: Suppose that for some propositional formula $A(x_1,...,x_n)$, there is
a Turing program $P$ such that $P([\phi_{1}],...,[\phi_{n}])\downarrow=1$ iff $\mathbb{N}\models A(\phi_{1},...,\phi_{n})$ and otherwise $P([\phi_{1}],...,[\phi_{n}])\downarrow=0$ (where $[\phi]$ denotes the 
G\"odel number of $\phi$), does it follow
that the truth value of $A(\phi_{1},...,\phi_{n})$ is independent of $\phi_1,...,\phi_{n}$ and hence that $A$ is tautological or contradictory?
\end{abstract}


\section{Decidability in PA and related systems}

\begin{defini}
 Let $T$ be a theory. A propositional formula $A(x_1,...,x_n)$ is always decicable in $T$ iff $T$ decides every sentence of the form
$A(\phi_1,...,\phi_n)$, where $\phi_1,...,\phi_n$ are closed formulas (without free variables) in the language of $T$.
\end{defini}

We formulate our claims for the case of $PA$, but they can be transfered to arbitrary recursive axiom systems that allow G\"odel coding.

\begin{lemma}{\label{contradictoryindependentformulas}}
 For each $n\in\omega$, there is a set of $n$ mutually exclusive non-refutable formulas, i.e. a set $\{\theta_1,...,\theta_{n}\}$ of closed $\mathcal{L}_{PA}$ formulas such that no $\neg\theta_{i}$ is provable in $PA$
and such that $\theta_{i}\rightarrow\bigwedge_{j=1,j\neq i}^{n}\neg\theta_{j}$ is provable in $PA$ for $i\in\{1,...,n\}$.
\end{lemma}
\begin{proof}
 We write $\phi<_{p}\psi$ for the statement `There is a $PA$-proof of $\neg\phi$ and the smallest G\"odel number $n$ of such a proof is smaller than the smallest G\"odel number of a proof of $\neg\psi$, provided there is one',
i.e. $\exists{x}(\text{Bew}(x,[\neg\phi])\wedge\forall{y<x}\neg\text{Bew}(y,\neg\psi))$, where $\text{Bew}(a,b)$ denotes `$a$ is the G\"odel number of a proof of the closed formula with G\"odel number $b$'.
Consider the following system of statements (we confuse formulas with their G\"odel numbers):\\
(1) $\bigwedge_{i=2}^{n}z_{1}<_{p}z_{i}$\\
(2) $\bigwedge_{i=1,i\neq 2}^{n}z_{2}<_{p}z_{i}$\\
...
(n) $\bigwedge_{i=1,i\neq n}^{n}z_{n}<_{p}z_{i}$\\
Applying the G\"odel fixpoint theorem generalized to $n$-tuples of formulas (see e.g. \cite{H}), we get statements $\theta_1,...,\theta_n$ such that
\begin{center}
 (*) $\theta_{i}\leftrightarrow\bigwedge_{j=1,j\neq i}^{n}\theta_{i}<_{p}\theta_{j}$
\end{center}
is provable in $PA$ for each $i\in\{1,2,..,n\}$. We claim that $\{\theta_{1},...,\theta_{n}\}$ is as desired.\\
First, if $\theta_{i}$ and $\theta_{j}$ are both true (where $i\neq j$), then there are by (*) (G\"odel numbers of) proofs $\beta_i$ for $\theta_{i}$ and $\beta_{j}$ for $\theta_{j}$.
Now, again by (*), we have $\beta_{i}<\beta_{j}$ and $\beta_{j}<\beta_{i}$, which is impossible. Hence $\theta_{i}$ implies $\neg\theta_{j}$ for all $j\neq i$. This argument can
easily be carried out in $PA$.\\
Second, suppose that $\neg\theta_{i}$ is provable in $PA$ for some $i\in\{1,...,n\}$.
 If $\neg\theta_{i}$ is provable, then there is $j\in\{1,...,n\}$ such that $\neg\theta_{j}$ is provable
and the minimal G\"odel number of a proof of $\neg\theta_{j}$ is minimal among the minimal G\"odel numbers of proofs of $\neg\theta_{k}$ for $k\in\{1,...,n\}$. Let $\beta$ be the minimal G\"odel number
of a proof of $\neg\theta_{j}$. Then $PA$ proves $\neg\theta_{j}$. Moroever, it is easily provable in $PA$ that no $k^{\prime}<k$ is a proof for any of the $\theta_{l}$, $l\in\{1,...,n\}$.
Hence, by (*), $PA$ proves $\theta_{j}$, so $PA$ proves $\theta_{j}\wedge\neg\theta_{j}$, a contradiction.\\
\end{proof}

\begin{lemma}{\label{independentformulas}}
 For each $n\in\omega$, there are $n$ formulas $\phi_{1},...,\phi_{n}$ in the language of arithmetic such that for no Boolean combination $C$ of any $n-1$ of them, $PA+C$ decides the remaining one.
\end{lemma}
\begin{proof}
Let $n\in\omega$. By Lemma \ref{contradictoryindependentformulas}, pick a set $S:=\{\theta_1,...,\theta_{2^{n}}\}$ of $2^{n}$ non-refutable, mutually exclusive formulas. We will construct
$\phi_{1},...,\phi_{n}$ as disjunctions $\bigvee{R}$ over subsets of $S$. By choice of the $\theta_{i}$, it is clear that $\theta_{i}\implies\bigvee{R}$ iff $\theta_{i}\in R$:
Clearly, if $\theta_{i}\in R$, then $\theta_{i}\implies\bigvee{R}$; on the other hand, $\theta_{i}$ implies $\neg\theta_{j}$ for all $j\neq i$, so $\theta_{i}\implies\neg\bigvee{R}$
if $\theta_{i}\notin R$.\\
Let $f$ be some bijection between $\mathfrak{P}(\{1,2,...,n\})$ and $S$. We proceed to define subsets $S_{1},...,S_{n}$ of $S$ as follows: We put $\theta_{i}$ in $S_{j}$
iff $j\in f^{-1}(\theta_{j})$. Hence each subset of $\{1,2,...,n\}$ is `marked' as the set of $j$ for which $S_{j}$ contains a particular $\theta_{i}$.
Set $\phi_{i}:=\bigvee{S_{j}}$. We claim that $\{\phi_{i}|1\leq i\leq n\}$ is as desired.\\
To see this, consider a combination $\bigwedge_{i=1}^{n}\delta_{i}\phi_{i}$ where each $\delta_{i}$ is either $\neg$ or nothing (i.e. each $\phi_{i}$ appears once, either plain or negated).
Then $E:=\{i|1\leq i\leq n\wedge \delta_{i}\neq\neg\}$ is a subset of $\{1,...,n\}$. Let $\theta_{j}=f(E)$. Then by what we just observed, $\theta_{j}$ implies all elements of $E$
and implies the negation of all elements of $S\setminus E$. Hence $\theta_{j}$ implies $\bigwedge_{i=1}^{n}\delta_{i}\phi_{i}$ (and this implication is provable in $PA$). 
Now, if $PA+\bigwedge_{i=1}^{n}\delta_{i}\phi_{i}$ was inconsistent, so was $PA+\theta_{j}$. But then, $PA$ would prove $\neg\theta_{j}$, contradicting the choice of $\theta_{j}$.
Hence $PA+\bigwedge_{i=1}^{n}\delta_{i}\phi_{i}$ is consistent. As $\bigwedge_{i=1}^{n}\delta_{i}\phi_{i}$ was arbitrary, $\{\phi_{1},..,\phi_{n}\}$ is indeed as desired.
\end{proof}

\textbf{Remark}: This is a generalization of a construction for the case $n=2$ given in \cite{P} (p. $19$), there attributed to E. Jerabek.

\begin{defini}
 A set $S$ of closed $\mathcal{L}_{PA}$-formulas is independent iff for no finite $S^{\prime}\subseteq S$, $\phi\in S\setminus S^{\prime}$ and no Boolean combination $C$ of $S^{\prime}$,
$PA+C$ decides $\phi$.
\end{defini}

\begin{lemma}{\label{allcombinationsconsistent}}
 If $S$ is a finite set of closed $\mathcal{L}_{PA}$ formulas, then $S$ is and independent over $PA$, iff for every Boolean combination $C$ of the elements of $S$ (conjunction in which each element of $S$ appears once, either plain or negated),
$PA+C$ is consistent (provided $PA$ is consistent).
\end{lemma}
\begin{proof}
 If some combination $C$ was inconsistent and $\psi_1,...,\psi_{n-1}$ were the first $n-1$ conjuncts of $C$ (i.e. $\phi_1,...,\phi_{n-1}$, either plain or negated), then $\phi_{n}$
would be decided by $\psi_1,...,\psi_{n-1}$, contradicting the assumption of independence.
\end{proof}

\begin{thm}{\label{maintheorem}}
Every always decidable formula is either tautological or contradictory, i.e.: Let $A(x_{1},...,x_{n})$ be a propositional formula in $n$ propositional variables $x_1,...,x_n$. Assume that for each $n$-tuple of $\mathcal{L}_{PA}$-formulas without free variables $(\phi_{1},...,\phi_{n})$,
we have that $PA$ decides $A(\phi_{1},...,\phi_{n})$ (i.e. $PA$ either proves the sentence or refutes it). Then $A$ is either a tautology or contradictory.
\end{thm}
\begin{proof}
 Write $A$ in disjunctive normal form. Supppose $A$ is neither tautological nor contradictory. Let $B_{1}:\{x_1,...,x_n\}\rightarrow\{0,1\}$ be an assignment of truth values to the proposition variables
that makes $A$ true and $B_{2}:\{x_{1},...,x_{n}\}\rightarrow\{0,1\}$ another one that makes it false. By Lemma \ref{independentformulas}, let $\{\phi_1,...,\phi_n\}$ be an independent set of $\mathcal{L}_{PA}$-formulas of
cardinality $n$. Let $C_1$ and $C_2$ be the Boolean combinations corresponding to $B_1$ and $B_2$, respectively. Then $PA+C_1$ and $PA+C_2$ are both consistent by Lemma \ref{allcombinationsconsistent}; however,
in $PA+C_1$, $A(\phi_1,...,\phi_{n})$ is true and in $PA+C_2$, $A(\phi_1,...,\phi_{n})$ is false. Hence $A(\phi_1,...,\phi_{n})$ is not decidable in $PA$, contradicting the assumption.
\end{proof}


\section{Algorithmical Decidability of Propositional Forms}

We ask a question analogous to that of the preceeding section, where decidability is now taken to mean decidability by a Turing machine:
Suppose that for some propositional formula $A(x_1,...,x_n)$, there is
a Turing program $P$ such that $P([\phi_{1}],...,[\phi_{n}])\downarrow=1$ iff $\mathbb{N}\models A(\phi_{1},...,\phi_{n})$ and otherwise $P([\phi_{1}],...,[\phi_{n}])\downarrow=0$, does it follow
that the truth value of $A(\phi_{1},...,\phi_{n})$ is independent of $\phi_1,...,\phi_{n}$ and hence that $A$ is tautological or contradictory?
It turns out that the answer is yes:

\begin{thm}{\label{TuringDec}}
Let $A$ be a propositional form and let $P$ be a Turing program such that $P([\phi_{1}],...,[\phi_{n}])\downarrow=1$ iff $\mathbb{N}\models A(\phi_{1},...,\phi_{n})$ and otherwise $P([\phi_{1}],...,[\phi_{n}])\downarrow=0$.
Then $A$ is tautological or contradictory.
\end{thm}
\begin{proof}
Assume that $P$ is such a program for a propositional formula $A$. We build a recursive extension $T$ of $PA$ that can roughly be stated as $PA+`P$ is always right'. 
As `$P$ is always right' is true by assumption, $T$ is consistent.
$T$ consists of $PA$ together with the sentence 
$S_{(\phi_{1},..,\phi_{n})}:=(A(\phi_{1},...,\phi_{n})\rightarrow P([\phi_{1}],...,[\phi_{n}])\downarrow=1)\wedge (\neg A(\phi_{1},...,\phi_{n})\rightarrow P([\phi_{1}],...,[\phi_{n}])\downarrow=0)$
for every $n$-tuple $(\phi_{1},...,\phi_{n})$ of closed formulas. Clearly, $T$ is recursive.\\
Now, as, by assumption, $P$ halts with output $0$ or $1$ on every $n$-tuple $(\phi_1,...,\phi_n)$ of closed formulas, $PA$ will prove this for every single instance;
moreoever, $T$ will, via the extra assumptions, know that $P$ decides correctly and hence decide $A(\phi_1,...,\phi_{n})$ for every such $n$-tuple. As Theorem \ref{maintheorem}
is valid for recursive extensions of $PA$, it is valid for $T$, so $A$ is either a tautology or contradictory, as desired.
\end{proof}



\end{document}